\documentclass[a4paper,11pt,reqno]{amsart}
\usepackage{amsmath}
\usepackage[T1]{fontenc}
\usepackage{amssymb}
\usepackage{amsthm}
\usepackage{color}
\usepackage[lmargin=2.5 cm,rmargin=2.5 cm,tmargin=3.5cm,bmargin=2.5cm,
paper=a4paper]{geometry}

\newcommand{\Ab}{\mathbf A}

\newcommand{\R}{\mathbb R}

\newcommand{\E}{\mathrm{E}_{\rm gs}(\kappa, H)}

\DeclareMathOperator{\curl}{curl}

\newtheorem{thm}{Theorem}[section]

\newtheorem{lem}[thm]{Lemma}

\theoremstyle{remark}

\numberwithin{equation}{section}

\title[Energy of Bulk superconductivity]
{A new formula for the energy of bulk superconductivity}

\author{Ayman Kachmar}
\address[A. Kachmar]{Department of Mathematics, Lebanese University, Hadat, Lebanon}
\email{ayman.kashmar@gmail.com}
\date{\today}

\begin{document}

\begin{abstract}
The energy of a type~II superconductor submitted to an external magnetic field of intensity close to the second critical field is given by the celebrated Abrikosov energy. If the external magnetic field is comparable to and  below the second critical field, the energy is given by a reference function obtained as a special (thermodynamic) limit of a non-linear energy. In this note, we give a new formula for this reference energy. In particular, we obtain it as a special  limit of a {\it linear} energy defined over  configurations  normalized in the $L^4$-norm.
\end{abstract}

\maketitle 

\section{Introduction}\label{hc2-sec:int}

\subsection{A background}

The behavior of a superconductor  subjected to an external magnetic field varies as the intensity of the applied magnetic field changes. That has been observed early in the physics literature  on theoretical and experimental grounds.   There are two important key observations regarding a specific class of materials, called type~II superconductors, namely, the   formation of Abrikosov lattices and the persistence of surface superconductivity. Abrikosov lattices occur when the intensity of the magnetic field is near a special value, called the   
{\it second critical field} and denoted by $H_{C_2}$. If the intensity of the external magnetic field is increased above this value, then superconductivity disappears from the bulk of the sample and remains on a (part) of the surface of the material. This phenomenon persists until the intensity of the applied magnetic field reaches another special value, called the {\it third critical field} and denoted by $H_{C_3}$. When the intensity of the applied magnetic field is increased further, superconductivity is destroyed everywhere in the sample, which is set into the normal state.  The reader may consult the book of de\,Gennes \cite{dG} for the description of these important observations. Both phenomena, {\it Abrikosov lattices} and {\it surface superconductivity} where observed by theoretically investigating the Ginzburg-Landau model. Ginzburg and Landau proposed the model on a phenomenological basis to describe the response of a superconductor to an external magnetic field.

Mathematically, the Ginzburg-Landau model is a functional defined on a certain class of configurations. The physically relevant states of the superconductor are those corresponding to minimizing configurations (and critical points) of the functional. The Abrikosov lattice is distinguished by a special behavior of the minimizing configurations. The same applies for the surface superconductivity phenomenon. The two monographs \cite{FH-b, SS-b} contain many mathematical results regarding the Ginzburg-Landau model  together with the discussion of their significance in physics.

\subsection{The Ginzburg-Landau model}

Here we describe the Ginzburg-Landau model for a superconducting sample, occupying an infinite cylindrical domain.  The cross section of the cylinder is assumed a smooth and simply connected open subset $\Omega$ of $\R^2$. 

The sample is  subjected to an external magnetic field with direction parallel to the axis of the cylinder. The intensity of the external magnetic field is assumed constant.

The superconducting material is distinguished by a characteristic parameter $\kappa>0$. When $\kappa$ is large, the material is of Type~II. The intensity of the external magnetic field is denoted by a parameter $h_{\rm ex}$.

The behavior of the superconductor is described by a {\it wave function} $\psi:\Omega\to\mathbb C$ and a {\it vector field} $\Ab:\Omega\to\R^2$.  The significance of $\psi$ and $\Ab$ is as follows. $|\psi|^2$ measures the density of the superconducting Cooper pairs, whose presence is necessary to observe  the superconductivity phenomenon,  and $\curl\Ab$ measures the induced magnetic field in the
sample, if present. At equilibrium, the configuration $(\psi,\Ab)$ minimizes the following energy, that we will call the Ginzburg-Landau energy:
\begin{equation}\label{eq-3D-GLf}
\mathcal E(\psi,\Ab)=\int_\Omega \left(|(\nabla-i\kappa H\Ab)\psi|^2+\frac{\kappa^2}2(1-|\psi|^2)^2+\kappa^2H^2|\curl\Ab-1|^2\right)\,dx\,,
\end{equation}
in the configurations space 
$$(\psi,\Ab)\in H^1(\Omega;\mathbb C)\times H^1(\Omega;\mathbb
R^2)\,,$$
where $H^1$ denotes the usual Sobolev space.

We introduce the ground state energy of the functional in
\eqref{eq-3D-GLf} as follows,
\begin{equation}\label{eq-gse}
\E=\inf\{\mathcal E(\psi,\Ab)~:~(\psi,\Ab)\in H^1(\Omega;\mathbb
C)\times H^1(\Omega;\mathbb R^2)\}\,.
\end{equation}
Here we use the version of the functional as in \cite{FH-b}. In this version, the intensity of the external magnetic field is measured by $h_{\rm ex}=\kappa H$, where $H>0$ is the the parameter marking the variation of the magnetic field. One way to detect the response of the superconductor to the external magnetic field is to mark changes in the energy as the parameter $H$ changes. A large part of the mathematical literature is devoted to the computation of \eqref{eq-gse} when $h_{\rm ex}$ is a function of $\kappa$ and $\kappa\to\infty$. The reader is referred  to  the two monographs \cite{FH-b,SS-b} for a detailed discussion  on the behavior of the energy in \eqref{eq-gse}. In two special regimes, the energy in \eqref{eq-gse} is given to leading order by the {\it Abrikosov} and {\it bulk} constants introduced below in \eqref{eq:Abcst} and \eqref{eq:enBLK} respectively.

\subsection{The Abrikosov energy}

We introduce the Abrikosov energy in a simple situation. More general situations are discussed in \cite{AfSe} but in the asymptotic limit considered here, they give rise to the same  Abrikosov constant in \eqref{eq:Abcst} below. Let $R>0$ and consider the lattice in $\R^2$ generated by the square
\begin{equation}\label{eq:QR}
Q_R=(-R/2,R/2)\times(-R/2,R/2)\,.
\end{equation}
Let us suppose that $R^2\in 2 \pi\mathbb N$. Denote by $P_R$ the self-adjoint operator
$$P_R=-(\nabla-i\Ab_0)^2\quad{\rm in~}L^2_{\rm mag, per}(Q_R)\,,$$
defined via the closed quadratic form
$$q_R(u)=\int_{Q_R}|(\nabla-i\Ab_0)u|^2\,dx\,.$$
Here, the vector potential $\Ab_0$ is  defined as follows,
\begin{equation}\label{eq:Ab0}
\Ab_0(x_1,x_2)=-\frac12(-x_2,x_1)\,,\end{equation}
and generates a unit constant magnetic field,
$$\curl\Ab_0=1\,.$$
The space $L^2_{\rm mag,per}(Q_R)$ and the form domain  $D(q_R)$ are defined as follows,
\begin{multline*}
L^2_{\rm mag,per}(Q_R)=\{u\in L^2_{\rm loc}(\R^2)~:~u(x_1+R,x_2)=e^{iRx_1/2}u(x_1,x_2)\\
{\rm ~and~}u(x_1,x_2+R)=e^{-iRx_2/2}u(x_1,x_2)\}\,,\end{multline*}
and
$$D(q_R)=H^1_{\rm mag,per}(Q_R):=\{u\in L^2_{\rm mag,per}(Q_R)~:~(\nabla-i\Ab_0)u\in L^2_{\rm loc}(\R^2)\}\,.$$
The space $L^2_{\rm mag,per}(Q_R)$ is a Hilbert space with the following inner product
$$\langle u,v\rangle_{L^2_{\rm mag,per}(Q_R)}=\int_{Q_R}u\,\overline{v}\,dx\,.$$
The spectrum of the operator $P_R$ is explicitly given by the Landau levels,
$$\sigma(P_R)=\{(2n-1)~:~n\in\mathbb N\}\,,$$
and, as long as $R^2\in2\pi\mathbb N$, all the eigenvalues have finite multiplicity.

Now, we introduce the Abrikosov energy functional (in the square lattice),
\begin{equation}\label{eq:enAb}
\mathcal E_{\rm Ab}(u)=\int_{Q_R}\left(\frac12|u|^4-|u|^2 \right)\,dx\,,\end{equation}
defined for configurations $u$ in the finite dimensional space $E_R:={\rm Ker}(P_R-{\rm Id})$, the first eigenspace of the operator $P_R$.
Minimizing the Abrikosov energy functional leads us to introduce the following quantity,
$$e_{\rm Ab}(R)=\inf\{\mathcal E_{\rm Ab}(u)~:~u\in E_R\}\,.$$
It is a known fact that \cite{AfSe, FK-cpde}:
\begin{equation}\label{eq:Abcst}
\lim_{\substack{R\to\infty\\R^2\in2\pi\mathbb N}}\frac{e(R)}{R^2} =E_{\rm Ab}\,.\end{equation}
Here, $E_{\rm Ab}$ is a universal constant that we will call the {\it Abrikosov constant} (or energy). It is known that $E_{\rm Ab}\in[-\frac12,0)$.

\subsection{The reference `bulk' energy}~

In this section, we recall how one can define the Abrikosov constant in \eqref{eq:Abcst} via a non-linear energy. Let $b\in(0,1]$ be a fixed constant, $R>0$ and $Q_R$ be the square in \eqref{eq:QR}. Define the following non-linear functional in $H^1(Q_R)$,
\begin{equation}\label{eq:enBLK}
\mathcal E_{\rm blk}(u;b,R)=\int_{Q_R}\left(b|(\nabla-i\Ab_0)u|^2-|u|^2+\frac12|u|^4\right)\,dx\,.\end{equation}
Here $\Ab_0$ is the magnetic potential in \eqref{eq:Ab0}. By minimizing this functional in various spaces, we get the following ground state energies,
\begin{align*}
&e^D(b,R)=\inf\{\mathcal E_{\rm blk}(u;b,R)~:~u\in H^1_0(Q_R)\}\,,\\
&e^N(b,R)=\inf\{\mathcal E_{\rm blk}(u;b,R)~:~u\in H^1(Q_R)\}\,,\\
&e^p(b,R)=\inf\{\mathcal E_{\rm blk}(u;b,R)~:~u\in H^1_{\rm mag,per}(Q_R)\}\,.
\end{align*}
It is a known fact that \cite{AfSe, Att, FK-cpde}
\begin{equation}\label{eq:Eblk}
\lim_{R\to\infty}\frac{e^{\circ}(b,R)}{R^2}=E_{\rm blk}(b)\quad (\circ\in\{D,N,p\})\,,
\end{equation}
where $E_{\rm blk}(\cdot)$ is a continuous and increasing function such that $E_{\rm blk}(0)=-\frac12$ and $E_{\rm blk}(1)=0$.
The function $E_{\rm blk}(\cdot)$ is independent of the boundary condition and will be called the reference bulk energy.

The Abrikosov constant in \eqref{eq:Abcst} can be defined in the alternative way \cite{AfSe, FK-cpde},
\begin{equation}\label{eq:Eblk-Ab}
E_{\rm Ab}=\lim_{b\to1_-}\frac{E_{\rm blk}(b)}{(b-1)^2}\,.
\end{equation}
This formula displays a relationship between the non-linear simplified Ginzburg-Landau energy in \eqref{eq:enBLK} and the Abrikosov energy in \eqref{eq:enAb}. 

\subsection{The connection with the full GL functional} 
To illustrate how the quantities discussed so far are useful, let us cite the following two results from \cite{SS02,FKhc2}. (We will use the following notation: If $a(\kappa)$ and $b(\kappa)$ are two non-negative functions of $\kappa$, then by writing $a(\kappa)\ll b(\kappa)$ we mean that $a(\kappa)=\delta(\kappa)b(\kappa)$ and $\delta(\kappa)\to0$ as $\kappa\to\infty$.)
\begin{enumerate}
\item If $b\in(0,1]$ is a constant, $H=b\kappa$ and $\kappa\to\infty$, then the ground state energy in \eqref{eq-gse} satisfies,
$$\E=\kappa^2|\Omega|E_{\rm blk}(b)+o(\kappa^2)\,.$$
\item If $H=\kappa-\mu(\kappa)$ and $\sqrt{\kappa}\ll\mu(\kappa)\ll \kappa$, then as $\kappa\to\infty$,
\begin{equation}\label{eq:gseFK}
\E=[\kappa-H]^2|\Omega|E_{\rm Ab}+o\big([\kappa-H]^2\big)\,.
\end{equation}
\end{enumerate}

\subsection{The new formula}
Here, we will define the function $E_{\rm blk}(\cdot)$ via a {\it non-linear} eigenvalue problem. Let again $b\in(0,1]$, $R>0$, $Q_R$ be the square in \eqref{eq:QR} and $\Ab_0$ be the magnetic potential in \eqref{eq:Ab0}. Let us define the {\rm linear} functional,
\begin{equation}\label{eq:enLin}
\mathcal E_{\rm lin}(u;b,R)=\int_{Q_R}\Big(b|(\nabla-i\Ab_0)u|^2-|u|^2\Big)\,dx\,.
\end{equation}
We will minimize this functional in various spaces but for the constrained configurations 
$$\int_{Q_R}|u|^4\,dx=1\,.$$
That way, we get the following ground state energies,
\begin{align}
&\mathfrak{m}^D(b,R)=\inf\Big\{\,
\frac{\mathcal E_{\rm lin}(u;b,R)}{\left(\displaystyle\int_{Q_R}|u|^4\,dx\right)^{1/2}}~:~u\in H^1_0(Q_R)\setminus\{0\}\,\Big\}\,,\label{eq:gse-const-D}\\
&\mathfrak{m}^N(b,R)=\inf\Big\{\,
\frac{\mathcal E_{\rm lin}(u;b,R)}{\left(\displaystyle\int_{Q_R}|u|^4\,dx\right)^{1/2}}~:~
{u\in H^1(Q_R)\setminus\{0\}}\,\Big\}\,,\label{eq:gse-const-N}\\
&\mathfrak{m}^p(b,R)=\inf\Big\{
\frac{\mathcal E_{\rm lin}(u;b,R)}{\left(\displaystyle\int_{Q_R}|u|^4\,dx\right)^{1/2}}~:~
u\in H^1_{\rm mag,per}(Q_R)\setminus\{0\}\,\Big\}\,.\label{eq:gse-const-p}
\end{align}
We will prove that,
$$\lim_{R\to\infty}\frac{\mathfrak{m}^{\circ}(b,R)}{R}=E_{\rm new}(b)\quad (\circ\in\{D,N,p\})\,,$$
and that 
\begin{equation}\label{eq:Enew=Eblk}
E_{\rm new}(b)=-\sqrt{-2E_{\rm blk}(b)}\,.
\end{equation}
More precisely:

\begin{thm}\label{thm:newF}
Let $b\in(0,1)$. There exist two constants $C>0$ and $R_0>1$ such that, for all $R\geq R_0$ and $\circ\in\{D,N,p\}$,
$$-\big(-2E_{\rm blk}(b)\big)^{1/2}-\frac{C}{R}\big(-2E_{\rm blk}(b)\big)^{-1/2}\leq \frac{\mathfrak{m}^{\circ}(b,R)}{R}\leq 
-\big(-2E_{\rm blk}(b)\big)^{1/2}+\frac{C}{R}\,.$$
In particular, we may define the function $E_{\rm blk}(\cdot)$ via the formula,
$$E_{\rm blk}(b)=-\frac12\left(\lim_{R\to\infty}\frac{\mathfrak{m}^{\circ}(b,R)}{R}\right)^2\qquad\big(b\in(0,1)\big)\,,$$
with  $\circ\in\{D,N,p\}$.

For the Dirichlet boundary condition, the lower bound is as follows
$$ -\big(-2E_{\rm blk}(b)\big)^{1/2}\leq \frac{\mathfrak{m}^{D}(b,R)}{R},$$
and is valid for all $b\in(0,1]$ and $R\geq 1$. 
\end{thm}

Let us compare the various ground state energies  discussed so far. The Abrikosov functional in \eqref{eq:enAb} is defined via a simple expression that does not involve differentiation operations but is minimized in the non-trivial space of the ground state  eigenfunctions of the operator $P_R$. Among the remaining functionals we discussed, the expression of the non-linear  functional in \eqref{eq:enBLK} is  the most complicated, but it is minimized in the space of Sobolev functions (this space is less complicated than the space of the ground state eigenfunctions  of the operator $P_R$). The expression of the functional in \eqref{eq:enLin} is linear, but again this functional is minimized for constrained configurations. The three functionals serve in defining the Abrikosov  constant $E_{\rm Ab}$ in \eqref{eq:Abcst}, but each time one singles a simpler  expression of the functional,  a price is paid through a constraint in the definition of the ground state energy. 

Recently, there is a progress in the analysis of  {\it semi-classical}  non-linear eigenvalue problems with a magnetic field (cf. \cite{FR}). The result in Theorem~\ref{thm:newF} may fall in this  area as well.

The rest of the paper is decomposed into three sections. The proof of Theorem~\ref{thm:newF} occupies Sections~\ref{sec:ub} and \ref{sec:lb}. In Section~\ref{sec:ap}, we apply Theorem~\ref{thm:newF} to write a new proof of an important theorem by Almog devoted to the full Ginzburg-Landau functional (cf. \cite[Thm.~3.3]{Al}).

\section{Proof of Theorem~\ref{thm:newF}: Upper bound}\label{sec:ub}

Let $b\in(0,1)$. This section is devoted to the proof of the following inequality
\begin{equation}\label{eq:Enew<}
\mathfrak{m}^{\circ}(b,R)\leq -R\big(-2E_{\rm blk}(b)\big)^{1/2}+C\,,
\end{equation}
valid for and $\circ\in\{D,N,p\}$ and $R\geq R_0$, where $C>0$ and $R_0>1$ are two constants that depend on $b$. 

Along the proof of \eqref{eq:Enew<}, the following two lemmas are needed.

\begin{lem}\label{lem:FK}(\cite{FK-cpde} and \cite{Att2})\\
There exists a constant $C>0$ such that, for all $b\in(0,1]$, $R>1$ and $\circ\in\{D,N,p\}$,
$$E_{\rm blk}(b)-\frac{C}{R}\leq\frac{e^{\circ}(b,R)}{R^2}\leq E_{\rm blk}(b)+\frac{C}{R}\,. $$
For the Dirichlet boundary condition, the lower bound is
$$E_{\rm blk}(b)\leq \frac{e^D(b,R)}{R^2}\,.$$
\end{lem}

\begin{lem}\label{lem:FK-ap}
There exists a constant $C>0$ such that, for all $b\in(0,1]$, $R>1$ and $\circ\in\{D,N,p\}$, if $u_{b,R}$ is a minimizer of $e^{\circ}(b,R)$, then,
$$-2R^2E_{\rm blk}(b)-CR\leq \int_{Q_R}|u_{b,R}|^4\,dx\leq -2R^2E_{\rm blk}(b)+CR\,.$$
\end{lem}
\begin{proof}
The minimizer $u_{b,R}$ satisfies the following equation
$$-b(\nabla-i\Ab_0)^2u_{b,R}=(1-|u_{b,R}|^2)u_{b,R}\quad{\rm in ~} Q_{R}\,,$$
with adequate boundary conditions along the boundary of $Q_R$ (Dirichlet for $\circ=D$, Neumann for $\circ=N$ and magnetic periodic for $\circ=p$).

Multiplying the equation of $u_{b,R}$ by $\overline{u_{b,R}}$, integating over $Q_{R}$ then applying an integration by parts, we obtain after a rearrangement of the terms,
$$-\frac12\int_{Q_R}|u_{b,R}|^2\,dx=\mathcal E_{\rm blk}(u_{b,R})=e^\circ(b,R)\,.$$
Now, applying Lemma~\ref{lem:FK} to estimate $e^{\circ}(b,R)$, we get the conclusion in 
Lemma~\ref{lem:FK-ap}. 
\end{proof}

\begin{proof}[Proof of \eqref{eq:Enew<}]
Let $b\in(0,1]$, $R>1$ and $u_{b,R}$ be a minimizer of $e^{\circ}(b,R)$ for $\circ\in\{D,N,p\}$. The ground state energy  $e^{\circ}(b,R)$ is displayed right after introducing the bulk functional $\mathcal E_{\rm blk}$ in \eqref{eq:enBLK}.

We write using in particular the definition of $\mathfrak{m}^{\circ}(b,R)$,
\begin{align*}
e^{\circ}(b,R)&=\mathcal E_{\rm blk}(u_{b,R})\\
&\geq \mathfrak{m}^{\circ}(b,R)\left(\int_{Q_R}|u_{b,R}|^4\,dx\right)^{1/2}+\frac12\int_{Q_R}|u_{b,R}|^4\,dx\,.
\end{align*} 
Next, we estimate 
the $L^4$-norm of $u_{b,R}$ by using Lemma~\ref{lem:FK-ap} to write, for some universal constant $C>0$,
$$
e^{\circ}(b,R)\geq \mathfrak{m}^{\circ}(b,R)\Big(-2R^2\,E_{\rm blk}(b)-CR\Big)_+^{1/2}+\frac12\Big(-2R^2\,E_{\rm blk}(b)-CR\Big)\,.
$$
Now, we estimate $e^{\circ}(b,R)$ as in Lemma~\ref{lem:FK}, arrange the terms and get for a possibly new value of the constant $C>0$,
$$
R^2E_{\rm blk}(b)+CR\geq R\Big(-2\,E_{\rm blk}(b)-CR^{-1}\Big)_+^{1/2}\mathfrak{m}^{\circ}(b,R)-R^2\,E_{\rm blk}(b)-CR\,.
$$
Now we select $R_0$ sufficiently large (depending on $b$) such that the term $-2\,E_{\rm blk}(b)-CR^{-1}$ is always positive for $R\geq R_0$, then we divide both sides by $R\Big(-2\,E_{\rm blk}(b)-CR^{-1}\Big)^{1/2}$ and rearrange the terms above to get \eqref{eq:Enew<}.
\end{proof}

\section{Proof of Theorem~\ref{thm:newF}: Lower bound}\label{sec:lb}

This section is devoted to the proof of the following inequality
\begin{equation}\label{eq:Enew>}
\mathfrak{m}^{\circ}(b,R)\geq -R\big(-2E_{\rm blk}(b)\big)^{1/2}-C\big(-2E_{\rm blk}(b)\big)^{-1/2}\,,
\end{equation}
valid for some universal constant $C>0$ and for all $b\in(0,1)$, $R>1$ and $\circ\in\{D,N,p\}$.

\begin{proof}[Proof of \eqref{eq:Enew>}]
Let $b\in(0,1)$, $R>1$  and $w_{b,R}$  be a minimizer of $\mathfrak{m}^{\circ}(b,R)$ for $\circ\in\{D,N,p\}$. The ground state energy  $\mathfrak{m}^{\circ}(b,R)$ is displayed right after introducing the bulk functional $\mathcal E_{\rm lin}$ in \eqref{eq:enLin}.

Let us normalize $w_{b,R}$ in $L^4$ as follows,
$$w^*_{b,R}=\frac{\big(-2R^2E_{\rm blk}(b)\big)^{1/4}}{\|w_{b,R}\|_{L^4(Q_R)}}w_{b,R}\,.$$
The $L^4$-norm of the normalized function satisfies 
$$\|  w^*_{b,R}\|_{L^4(Q_R)}=\big(-2R^2E_{\rm blk}(b)\big)^{1/4}\,.$$
By definition of the functional in \eqref{eq:enLin} and $\mathfrak{m}^{\circ}(b,R)$, we see that,
\begin{equation}\label{eq:m-en>}
\mathfrak{m}^{\circ}(b,R)=\frac{\mathcal E_{\rm lin}(w_{b,R})}{\|w_{b,R}\|_{L^4(Q_R)}^2}=\big(-2R^2E_{\rm blk}(b)\big)^{-1/2}\mathcal E_{\rm lin}(w^*_{b,R})\,.\end{equation}
Now,  we write using in particular the normalization of $w^*_{b,R}$ and the  definition of $e^{\circ}(b,R)$, 
\begin{align*}
\mathcal E_{\rm lin}(w^*_{b,R})&=\mathcal E_{\rm blk}(w^*_{b,R})-\frac12\int_{Q_R}|w^*_{b,R}|^4\,dx\\
&=\mathcal E_{\rm blk}(w^*_{b,R})+R^2E_{\rm blk}(b)\\
&\geq e^{\circ}(b,R)+R^2E_{\rm blk}(b)\,.
\end{align*} 
We estimate $e^{\circ}(b,R)$ from below using Lemma~\ref{lem:FK} to obtain,
$$\mathcal E_{\rm lin}(w^*_{b,R})\geq 2R^2E_{\rm blk}(b)-CR\,.$$
We insert this into \eqref{eq:m-en>} to get the inequality in \eqref{eq:Enew>}. The improved lower bound for the Dirichlet boundary condition holds since (cf. Lemma~\ref{lem:FK})
$$e^D(b,R)\geq R^2E_{\rm blk}(b)\,.$$
Also, this last bound trivially holds for $b=1$, since $E_{\rm blk}(1)=0$ and the spectral theory of the magnetic Laplacian with a Dirichlet boundary condition yields that $e^D(1,R)\geq 0$.
\end{proof}

\section{Application: Almog's $L^4$-bound}\label{sec:ap}

In \cite{Al}, Almog estimates the $L^4$-norm of the Ginzburg-Landau order parameter for the three dimensional functional (the proof is valid for the two dimensional functional as well). This bound is of particular importance to estimate the error terms when seeking a fine approximation of the ground state energy, as in \cite{FK-cpde, FKP}.

Here, we consider the same question as in \cite{Al} but in two dimensions. The method we give works in three dimensions as well, but we restrict to two dimensions for the sake of simplicity. 

We prove the following theorem:

\begin{thm}\label{thm:l4}
Let $\Lambda\in(0,1)$. There exist two constants $C>0$ and $\kappa_0$ such that, if $\kappa\geq \kappa_0$, $\Lambda\kappa\leq H\leq \kappa$, and $(\psi,\Ab)_{\kappa,H}$ is a critical point of \eqref{eq-3D-GLf},  then
\begin{equation}\label{eq:l4}
\frac1{|Q_\kappa|}\int_{Q_\kappa}|\psi|^4\,dx\leq \frac{C}{\kappa}+C\left(\frac{H}{\kappa}-1\right)^2\,,
\end{equation}  
where $Q_\kappa\subset \Omega$ is any square of side-length $2\kappa^{-1/2}$ and satisfying
$$\overline{Q_\kappa}\subset\big\{{\rm dist}(x,\partial\Omega)> 2\kappa^{-1/2}\big\}.$$
\end{thm}

We stress again that  the estimate in \eqref{eq:l4} is proved in \cite{Al} for three dimensional domains and when $Q_\kappa=\Omega$. The proof we give to Theorem~\ref{thm:l4} is based on Theorem~\ref{thm:newF} and differs from the one used in \cite{Al}.

Sharper versions of the bound in \eqref{eq:l4} are given in \cite{FKhc2, K-sima}. However, these improved versions of \eqref{eq:l4} are based on energy expansions of the form in \eqref{eq:gseFK}. The proof of \eqref{eq:gseFK} requires a rough control of the order parameter similar to the one in \eqref{eq:l4}. In \cite{FKhc2}, a key element in the proof of \eqref{eq:gseFK} was a strong $L^\infty$ bound on the order parameter, namely (cf. \cite{FH-jems, FKhc2})
\begin{equation}\label{eq:linfty}
\|\psi\|_{L^\infty(\omega)}\leq \frac{C}{\kappa^{1/2}}+C\left(\frac{H}{\kappa}-1\right)^{1/2}\,,
\end{equation}
where $\omega\subset\subset\Omega$. Having \eqref{eq:l4} in hand, we can derive the energy expansion in \eqref{eq:gseFK} without using the $L^\infty$ bound in \eqref{eq:linfty}.

Unlike the hard proof of \eqref{eq:linfty}, the proof we give to \eqref{eq:l4} seems quite general, does not require too much regularity of the critical points and works for functionals having a similar structure as that in \eqref{eq-3D-GLf}, e.g. the functional with a  variable magnetic field or with a pinning term (cf. \cite{Att, Att3}).

The rest of this section is devoted to the proof of Theorem~\ref{thm:l4}. We will use $C$ to denote positive constants independent of $\kappa$ and $H$. The value of $C$ might change from one formula to another without  explicit notice. 

\subsection{Preliminaries}

A critical point $(\psi,\Ab)$ of the functional in \eqref{eq-3D-GLf} is a solution of the following system of PDE:
\begin{equation}\label{eq:GL*}
\left\{
\begin{array}{ll}
-(\nabla-i\kappa H\Ab)^2\psi=\kappa^2(1-|\psi|^2)\psi\,,\\
-\nabla^\bot\curl\Ab=(\kappa H)^{-1}{\rm Im}(\overline{\psi}\,(\nabla-i\kappa H\Ab)\psi)\,,\quad{\rm in}~\Omega\,,\\
\nu\cdot(\nabla-i\kappa H\Ab)\psi=0\,,\quad \curl\Ab=1\quad{\rm on~}\partial\Omega\,.
\end{array}
\right.
\end{equation}

We collect useful {\it a priori} estimates  in:
\begin{lem}\label{lem:EE}
Let $\Lambda\in(0,1)$. There exist two constants $C>0$ and $\kappa_0>0$ such that, if $\kappa\geq \kappa_0$, $\Lambda\kappa\leq H\leq \kappa$, and $(\psi,\Ab)_{\kappa,H}$ is a solution of \eqref{eq:GL*}, then,
$$\|\psi\|_\infty\leq 1\,,$$
and
$$\|\curl\Ab-1\|_{C^1(\overline{\Omega})}\leq C\kappa^{-1}\,.$$
\end{lem}

We refer the reader to \cite{FH-b} for the proof of Lemma~\ref{lem:EE}. Based on the estimates in Lemma~\ref{lem:EE}, one can construct the  gauge transformation given in the next lemma (cf. \cite[Eq.~(5.30)]{FKhc2}):

\begin{lem}\label{lem:gt}
Let $\Lambda\in(0,1)$. There exist two constants $C>0$ and $\kappa_0>0$ such that the following is true. 

Suppose that $\kappa\geq \kappa_0$, $\Lambda\kappa\leq H\leq \kappa$,  and $(\psi,\Ab)_{\kappa,H}$ is a solution of \eqref{eq:GL*}. Let $\ell\in(0,1)$ and $B_\ell\subset\Omega$ be a disk of radius $2\ell$. There exists a function $\phi\in H^2(B_\ell)$ such that,
\begin{equation}\label{eq:gt}\forall~x\in\overline{B_\ell}\,,\quad\big|\Ab(x)-\big(\Ab_0(x)-\nabla\phi(x)\big)\big|\leq C\kappa^{-1}\ell\,.\end{equation}
Here, $\Ab_0$ is the magnetic potential in \eqref{eq:Ab0}.
\end{lem}

\subsection{Local estimates - Proof of Theorem~\ref{thm:l4}}

Here we work under the assumptions in Theorem~\ref{thm:l4}. We will estimate the following local energy of the critical configuration $(\psi,\Ab)$:
\begin{equation}\label{eq:gse-loc}
\mathcal E_{0,\kappa}(\psi,\Ab)=\int_{Q_{\kappa}}\left(|(\nabla-i\kappa H\Ab)\psi|^2-\kappa^2|\psi|^2+\frac{\kappa^2}2|\psi|^4\right)\,dx\,.
\end{equation}
Let $Q_{2\kappa}$ be the square having the same center as $Q_\kappa$ but with side-length $4\kappa^{-1/2}$, i.e. twice the side-length of $Q_\kappa$. Obviously the square $Q_{2\kappa}$ contains $\overline{Q_\kappa}$. Let  $f\in C_c^\infty(Q_{2\kappa})$ be a cut-off function satisfying, for all $\kappa\geq 1$,
$$f=1\quad{\rm in~} Q_{\kappa},\quad 0\leq f\leq 1{~\rm and~}|\nabla f|\leq C\kappa^{1/2}\quad {\rm in~}Q_{2\kappa}\,,$$
where $C$ is a constant independent of $\kappa$.

An integration by parts and the first equation in \eqref{eq:GL*} yield (cf. \cite[Eq.~(6.2)]{FKhc2}),
\begin{equation}\label{eq:decomp*}
\begin{aligned}
\mathcal  E_{0,2\kappa}(f\psi,\Ab)&=\kappa^2\int_{Q_{2\kappa}}f^2\left(-1+\frac12f^2\right)|\psi|^4\,dx+\int_{Q_{2\kappa}}|\nabla f|^2|\psi|^2\,dx\\
&\leq C\,.
\end{aligned}
\end{equation}
Note that we dropped the term involving  $-1+\frac12f^2$ since $0\leq f\leq 1$. The term involving $|\nabla f|$ is estimated using the bounds $\|\psi\|_\infty\leq 1$, $|\nabla f|\leq C\kappa^{1/2}$ and $|Q_{2\kappa}|\leq C\kappa^{-1}$.

Now we estimate the linear energy
\begin{equation}\label{eq:linGL}
\mathcal L_{0,\kappa}(f\psi,\Ab)=  \int_{Q_{2\kappa}}\Big(|(\nabla-i\kappa H\Ab)f\psi|^2-\kappa^2|f\psi|^2\Big)\,dx\,.
\end{equation}
Choose $C>0$ large enough such that $Q_{2\kappa}\subset B_{C\kappa^{-1/2}}$, then apply Lemma~\ref{lem:gt} in $B_{C\kappa^{-1/2}}$ to get a function $\phi\in H^2(B_{C\kappa^{-1/2}})$ satisfying the estimate in \eqref{eq:gt}. Notice that, using the Gauge invariance then the Cauchy-Schwarz inequality,
\begin{align*}
\mathcal L_{0,\kappa}(f\psi,\Ab)&=\mathcal L_{0,\kappa}\Big(f\psi e^{-i\kappa H\phi},\Ab-\nabla\phi\Big)\\
&\geq \int_{Q_{2\kappa}}\Big((1-\kappa^{-1/2})|(\nabla-i\kappa H\Ab_0)f\psi e^{-i\kappa H\phi}|^2-\kappa^2|f\psi e^{-i\kappa H\phi}|^2-C\kappa^{3/2}|f\psi|^2\Big)dx\,.
\end{align*} 
Let $b=(1-\kappa^{-1/2})\frac{H}{\kappa}$, $R=\kappa^{-1/2}\sqrt{\kappa H}=\sqrt{H}$ and $x_\kappa$ be the center of the square $Q_{2\kappa}$. Apply the change of variable $y=\sqrt{\kappa H}\,(x-x_\kappa)$ to get,
$$
\mathcal L_{0,\kappa}(f\psi,\Ab)\geq \kappa^{3/2}H^{-1/2}\mathfrak m^D(b,R)\|f\psi\|_4^2
-C\kappa^{3/2}\|f\psi\|_2^2\,,$$
where $\mathfrak m^D(b,R)$ is the energy introduced in \eqref{eq:gse-const-D}. We use the lower bound for $\mathfrak m^D(b,R)$ in Theorem~\ref{thm:newF} and H\"older's inequality for the term  $\|f\psi\|_2$ to get,
$$
\mathcal L_{0,\kappa}(f\psi,\Ab)\geq \kappa^{3/2}H^{-1/2}\Big(-\big(-2E_{\rm blk}(b)\big)^{1/2}\Big)R\|f\psi\|_4^2
-C\kappa\|f\psi\|_4^2\,.$$
Recall that $R=\sqrt{H}$ and insert the result into \eqref{eq:linGL} and the right side of \eqref{eq:decomp*} to get, after a rearrangement of the terms,
\begin{equation}\label{eq:decomp**}
\kappa^{3/2}\Big\{\Big(-|2E_{\rm blk}(b)|^{1/2}-C\kappa^{-1/2}\Big)+\frac{\kappa^{1/2}}2\|f\psi\|_4^2\Big\}\|f\psi\|_4^2
\leq C\,.
\end{equation}
Two cases may occur:
\begin{itemize}
\item {\bf Case~I:}
$$\Big(-|2E_{\rm blk}(b)|^{1/2}-C\kappa^{-1/2}\Big)+\frac{\kappa^{1/2}}2\|f\psi\|_4^2\leq \kappa^{-1/2}\,.$$
\item {\bf Case~II:}
$$\Big(-|2E_{\rm blk}(b)|^{1/2}-C\kappa^{-1/2}\Big)+\frac{\kappa^{1/2}}2\|f\psi\|_4^2\geq \kappa^{-1/2}\,.$$
\end{itemize}
Clearly, in both cases, \eqref{eq:decomp**} yields the following upper bound:
\begin{equation}\label{eq:l4ub}
\|f\psi\|_4^2\leq 2\kappa^{-1/2}|2E_{\rm blk}(b)|^{1/2}+C\kappa^{-1}\,.
\end{equation}
Since $f=1$ in $Q_\kappa$, then \eqref{eq:l4ub} says that
\begin{equation}\label{eq:l4ub*}
\left(\int_{Q_\kappa}|\psi|^4\,dx\right)^{1/2}\leq 2\kappa^{-1/2}|2E_{\rm blk}(b)|^{1/2}+C\kappa^{-1}\,.
\end{equation}
Recall that $b=(1-\kappa^{-1/2})\frac{H}{\kappa}\leq 0$. The assumption in Theorem~\ref{thm:l4} and the formula in \eqref{eq:Eblk-Ab} together yield that
$$|2E_{\rm blk}(b)|\leq C|b-1|\leq C\left|\frac{H}\kappa-1\right|+C\kappa^{-1/2}\,.$$
Inserting this into \eqref{eq:l4ub*} and remembering that $|Q_\kappa|=2\kappa^{-1}$, we get the estimate in \eqref{eq:l4}. This finishes the proof of Theorem~\ref{thm:l4}. 

\subsection*{Acknowledgments} The author is supported by the Lebanese University research funds.

\end{document}